\documentclass[12pt, a4paper, oneside]{amsart}
\usepackage{amsmath, amsthm, amssymb, geometry, mathrsfs}
\usepackage{tikz-cd}
\usepackage{hyperref}
\usepackage[shortlabels, inline]{enumitem}

\newtheorem{theorem}{Theorem}[section]

\newtheorem{lemma}[theorem]{Lemma}
\newtheorem{corollary}[theorem]{Corollary}

\theoremstyle{definition}
\newtheorem{definition}[theorem]{Definition}

\newcommand\B{\mathbb{B}}
\newcommand\C{\mathbb{C}}

\newcommand\N{\mathbb{N}}
\renewcommand\P{\mathbb{P}}

\newcommand{\cO}{\mathcal{O}}

\newcommand{\der}{\mathrm{d}}

\newcommand{\id}{\mathrm{id}}
\newcommand{\pr}{\mathrm{pr}}

\newcommand{\T}{\mathrm{T}}

\DeclareMathOperator{\Aut}{Aut}

\DeclareMathOperator{\Bl}{Bl}

\newcommand{\Ell}{\mathrm{Ell}}

\begin{document}

\title[Oka complements of countable sets and non-elliptic Oka manifolds]{Oka complements of countable sets \\ and non-elliptic Oka manifolds}
\author{Yuta Kusakabe}
\address{Department of Mathematics, Graduate School of Science, Osaka University, Toyonaka, Osaka 560-0043, Japan}
\email{y-kusakabe@cr.math.sci.osaka-u.ac.jp}

\begin{abstract}
We study the Oka properties of complements of closed countable sets in $\C^{n}\ (n>1)$ which are not necessarily discrete.
Our main result states that every tame closed countable set in $\C^{n}\ (n>1)$ with a discrete derived set has an Oka complement.
As an application, we obtain non-elliptic Oka manifolds which negatively answer a long-standing question of Gromov.
Moreover, we show that these examples are not even weakly subelliptic.
It is also proved that every finite set in a Hopf manifold has an Oka complement and an Oka blowup.
\end{abstract}

\maketitle

%
%

\section{Introduction}

A complex manifold $Y$ is an \emph{Oka manifold} if any holomorphic map from an open neighborhood of a compact convex set $K\subset\C^n\ (n\in\N)$ to $Y$ can be uniformly approximated on $K$ by holomorphic maps $\C^n\to Y$.
The most basic examples of Oka manifolds are complex Euclidean spaces $\C^{n}\ (n\in\N)$.
This is a consequence of the classical Oka-Weil approximation theorem which generalizes the Runge approximation theorem.
For the historical background and the theory of Oka manifolds, we refer the reader to the survey \cite{Forstneric2013} and the comprehensive monograph \cite{Forstneric2017}.

It is natural to ask when a closed countable set $S$ in $\C^n$ has an Oka complement $\C^{n}\setminus S$.
In the case of $n=1$, we have a complete answer that the complement $\C\setminus S$ is Oka if and only if $S$ contains at most one point.
On the other hand, in the case of $n>1$, Forstneri\v{c} and Prezelj proved that every \emph{tame} discrete set in $\C^n$ has an Oka complement \cite[Theorem 1.6]{Forstneric2002}.
The following is the definition of tameness.

\begin{definition}
A closed countable set $S\subset\C^{n}$ is \emph{tame} if there exists a holomorphic automorphism $\varphi\in\Aut\C^{n}$ such that the closure of $\varphi(S)$ in the projective space $\P^{n}\supset\C^{n}$ does not contain the hyperplane at infinity $\P^{n}\setminus\C^{n}$.
\end{definition}

For example, every compact countable set $S\subset\C^{n}$ is tame.
Tameness for discrete sets was first introduced by Rosay and Rudin \cite{Rosay1988} and recently generalized to other complex manifolds \cite{Andrist2019,Winkelmanna,Winkelmann2019}.
Rosay and Rudin also constructed a discrete set in $\C^{n}\ (n>1)$ whose complement is not Oka \cite[Theorem 4.5]{Rosay1988}.

In the present paper, we study complements of tame closed countable sets $S\subset\C^n$ which are \emph{not necessarily discrete}.
That is, we do not assume that the derived set $S'$ (the set of accumulation points of $S$) is empty.
Our main theorem is the following generalization of the result for tame discrete sets.

\begin{theorem}
\label{theorem:main}
For any tame closed countable set $S\subset\C^{n}\ (n>1)$ with a discrete derived set, the complement $\C^{n}\setminus S$ is Oka.
\end{theorem}

In the same manner, we may prove the following theorem for blowups.

\begin{theorem}
\label{theorem:blowup}
For any tame closed countable set $S\subset\C^{n}\ (n>1)$ with a discrete derived set $S'$, the blowup $\Bl_{S\setminus S'}(\C^{n}\setminus S')$ of $\C^{n}\setminus S'$ along $S\setminus S'$ is Oka.
\end{theorem}

The proofs of Theorem \ref{theorem:main} and Theorem \ref{theorem:blowup} are given in Section \ref{section:proof_theorem}.
These are new applications of the localization principle for Oka manifolds (Theorem \ref{theorem:localization}) which was established in our previous paper \cite{Kusakabe2019}.

In Section \ref{section:proof_corollary}, we prove the following two corollaries.
First, we give an example of a non-elliptic Oka manifold.
\emph{Ellipticity} was introduced by Gromov in his seminal paper \cite{Gromov1989} in 1989 (see Definition \ref{definition:elliptic}).
It is a consequence of his main result in \cite{Gromov1989} that ellipticity is a sufficient condition for manifolds to be Oka.
In the same paper, he also proved the converse for Stein manifolds \cite[Remark 3.2.A]{Gromov1989}\footnote{The condition $\Ell_\infty$ in \cite[Remark 3.2.A]{Gromov1989} is equivalent to being Oka (cf. \cite[\S 5.15]{Forstneric2017}).}.
Then he asked a question whether the converse holds for \emph{all} complex manifolds \cite[Question 3.2.A$''$]{Gromov1989}.
Decades later, Andrist, Shcherbina and Wold \cite{Andrist2016} showed that in a Stein manifold of dimension at least three every compact holomorphically convex\footnote{The holomorphic convexity was not assumed in \cite{Andrist2016} but used in the proof implicitly.} set with an infinite derived set has a non-elliptic complement.
However, there has been no example of an Oka complement of this type.
As an application of Theorem \ref{theorem:main} and the localization principle for Oka manifolds (Theorem \ref{theorem:localization}), we obtain such an example which negatively answers the long-standing question of Gromov.

\begin{corollary}
\label{corollary:non-elliptic}
For any $n\geq 3$, the complement $\C^{n}\setminus((\overline{\N^{-1}})^{2}\times\{0\}^{n-2})$ is a non-elliptic Oka manifold where $\N^{-1}=\{j^{-1}:j\in\N\}\subset\C$.
\end{corollary}

There is a weaker variant of ellipticity called \emph{weak subellipticity} (see Definition \ref{definition:elliptic}) which is also a sufficient condition to be Oka (cf. \cite[Corollary 5.6.14]{Forstneric2017}).
In fact, we prove that the complement in Corollary \ref{corollary:non-elliptic} is not even weakly subelliptic (Corollary \ref{corollary:non-weakly_subelliptic}).
It should be mentioned that a still weaker variant, called \emph{Condition $\Ell_{1}$}, characterizes Oka manifolds \cite[Theorem 1.3]{Kusakabe2019}.
The localization principle for Oka manifolds (Theorem \ref{theorem:localization}) was proved as a corollary of this characterization.

It is also a fundamental problem whether any point in an Oka manifold of dimension at least two has an Oka complement and an Oka blowup.
As another application of our results, we solve this problem for Hopf manifolds (in particular, we solve \cite[Problem 2.42]{Forstneric2013}).
Recall that a complex manifold $Y$ is a \emph{Hopf manifold} if it is compact and universally covered by $\C^{\dim Y}\setminus\{0\}$.
The latter condition implies that a Hopf manifold is Oka (cf. \cite[Corollary 5.6.11]{Forstneric2017}) and the dimension of a Hopf manifold must be greater than $1$.

\begin{corollary}
\label{corollary:Hopf}
For any Hopf manifold $Y$ and any finite set $S\subset Y$, the complement $Y\setminus S$ and the blowup $\Bl_{S}Y$ are Oka.
\end{corollary}

%
%

\section{Proofs of the Theorems}
\label{section:proof_theorem}

Recall the following localization principle for Oka manifolds.
Here, a subset of $Y$ is Zariski open if its complement is a closed complex subvariety.

\begin{theorem}[{\cite[Theorem 1.4]{Kusakabe2019}}]
\label{theorem:localization}
Let $Y$ be a complex manifold.
Assume that each point of $Y$ has a Zariski open Oka neighborhood.
Then $Y$ is an Oka manifold.
\end{theorem}

We also need the following approximation theorem.
Let $\overline{\B^{n}(a,r)}$ denote the closed ball in $\C^{n}$ of radius $r>0$ centered at $a\in\C^{n}$.

\begin{lemma}
\label{lemma:approximation}
For any discrete sequence $\{a_{j}\}_{j}$ in $\C^{n}$ there exists a sequence $\{r_j\}_{j}$ of positive numbers such that
\begin{enumerate}
	\item the closed balls $\overline{\B^{n}(a_{j},r_{j})}$ are mutually disjoint, and
	\item for any holomorphic functions $f_{j}\in\cO(\overline{\B^{n}(a_{j},r_{j})})$ and any sequence $\{\varepsilon_{j}\}_{j}$ of positive numbers there exists a holomorphic function $f\in\cO(\C^n)$ such that $\sup_{\overline{\B^{n}(a_{j},r_{j})}}|f-f_{j}|\leq\varepsilon_j$ for all $j$.
\end{enumerate}
\end{lemma}

\begin{proof}
We write $\overline\B_{R}=\overline{\B^{n}(0,R)}$.
Take an increasing sequence $\{R_{l}\}_{l\in\N}$ of positive numbers such that $\lim_{l\to\infty}R_{l}=\infty$ and $\{a_{j}\}_{j}\cap\bigcup_{l\in\N}\partial\overline\B_{R_{l}}=\emptyset$.
We define $\overline\B_{R_{0}}=\emptyset$ for convenience.
Renumbering $\{a_{j}\}_{j}$ if necessary, we may assume that there exists an increasing sequence $0=k_{0}<k_{1}<k_{2}<\cdots$ of integers such that $\{a_{j}\}_{j=k_{l}+1}^{k_{l+1}}\subset\overline\B_{R_{l+1}}\setminus\overline\B_{R_{l}}$ for all $l\geq 0$.
Since $\overline\B_{R_{l}}\cap\{a_{j}\}_{j=k_{l}+1}^{k_{l+1}}=\emptyset$ for each $l$, there exists a sequence $\{r_j\}_{j}$ of small positive numbers such that
\begin{enumerate}[(a)]
	\item the closed balls $\overline{\B^{n}(a_{j},r_{j})}$ are mutually disjoint,
	\item $\bigcup_{j=k_{l}+1}^{k_{l+1}}\overline{\B^{n}(a_{j},r_{j})}\subset\overline\B_{R_{l+1}}\setminus\overline\B_{R_{l}}$ for all $l\geq 0$, and
	\item $\overline\B_{R_{l}}\cup\bigcup_{j=k_{l}+1}^{k_{l+1}}\overline{\B^{n}(a_{j},r_{j})}$ is polynomially convex for each $l\geq 0$.
\end{enumerate}

Let us verify the condition $(2)$.
Let $g_{0}\in\cO(\overline\B_{R_{0}})=\cO(\emptyset)$ be the unique holomorphic function on $\emptyset$.
Assume inductively that $g_{l}\in\cO(\overline\B_{R_{l}})$ has been chosen for some $l\geq 0$.
By the Oka-Weil approximation theorem, there exists a holomorphic function $g_{l+1}\in\cO(\overline\B_{R_{l+1}})$ such that
\begin{enumerate}[(i)]
	\item $\sup_{\overline\B_{R_{l}}}|g_{l+1}-g_{l}|\leq\min\{\varepsilon_{j}\}_{j=1}^{k_{l}}/2^{l+1}$, and
	\item $\sup_{\overline{\B^{n}(a_{j},r_{j})}}|g_{l+1}-f_{j}|\leq\varepsilon_{j}/2$ for all $j=k_{l}+1,\ldots,k_{l+1}$.
\end{enumerate}
Then the limit $f=\lim_{l\to\infty}g_{l}$ exists uniformly on compacts in $\C^{n}$ and has the desired property.
\end{proof}

\begin{proof}[Proof of Theorem \ref{theorem:main}]
By Theorem \ref{theorem:localization}, it suffices to prove that for any fixed point $p\in\C^{n}\setminus S$ there exists a Zariski open Oka neighborhood of $p$.
Since $S$ is tame, there exist a holomorphic coordinate system $z=(z',z_{n})$ on $\C^{n}$ and a constant $C>0$ such that $p=0$ and $S\subset\{(z',z_{n}):|z_{n}|\leq C(1+|z'|)\}$ by definition.
Furthermore, since $S$ is countable, we may also assume that $S\subset(\C^{n-1}\setminus\{0\})\times\C$ and $\pr'|_{S}:S\to\C^{n-1}$ is injective where we denote by $\pr':\C^{n}\to\C^{n-1}$ the projection $z\mapsto z'$.
Note that the restriction of $\pr'$ to $\{(z',z_{n}):|z_{n}|\leq C(1+|z'|)\}$ is proper.
Thus $\pr'(S')$ is discrete in $\C^{n-1}$.
Let us enumerate $\pr'(S')=\{a_{j}\}_{j}$.
Take a sequence $\{r_{j}\}_{j}$ of positive numbers which satisfies the conditions (1) and (2) in Lemma \ref{lemma:approximation}.
Note that $\pr'(S)\setminus\bigcup_{j}\overline{\B^{n-1}(a_{j},r_{j})}\subset\C^{n-1}$ is discrete.
By the Oka-Cartan extension theorem, there exists $g\in\cO(\C^{n-1})$ such that
\begin{enumerate}[(a)]
	\item $g(0)=1$, and
	\item $g(z')=-z_{n}$ for all $(z',z_{n})\in S'\cup (S\setminus\pr'^{-1}(\bigcup_{j}\overline{\B^{n-1}(a_{j},r_{j})}))$.
\end{enumerate}
By the condition (2) in Lemma \ref{lemma:approximation}, there exists $f\in\cO(\C^{n-1})$ such that on each closed ball $\overline{\B^{n-1}(a_{j},r_{j})}$ the real part of $f$ satisfies
\begin{align*}
\Re f\leq\log\frac{1}{j\cdot\sup_{z'\in\overline{\B^{n-1}(a_{j},r_{j})}}(C(1+|z'|)+|g(z')|)}.
\end{align*}
It follows that for all $(z',z_{n})\in S\cap\pr'^{-1}(\overline{\B^{n-1}(a_{j},r_{j})})$
\begin{align*}
\left|e^{f(z')}(z_{n}+g(z'))\right|\leq e^{\Re f(z')}(C(1+|z'|)+|g(z')|)\leq\frac{1}{j}.
\end{align*}

Consider the automorphism $\varphi\in\Aut\C^{n}$ defined by $\varphi(z',z_{n})=(z',e^{f(z')}(z_{n}+g(z')))$.
Note that $\varphi(p)=(0,e^{f(0)})\in\C^{n-1}\times\C^{*}$, $\varphi(S')\subset\C^{n-1}\times\{0\}$ and the discrete set $D=\varphi(S)\cap(\C^{n-1}\times\C^{*})$ in $\C^{n-1}\times\C^{*}$ is contained in $\bigcup_{j}(\overline{\B^{n-1}(a_{j},r_{j})}\times\overline{\B^{1}(0,1/j)})$.
Thus $\pr_{n}(D)$ is discrete in $\C^{*}$ where $\pr_{n}:\C^{n}\to\C$ is the $n$-th projection, and hence $\pr_{n}((\id_{\C^{n-1}}\times\exp)^{-1}(D))=\exp^{-1}(\pr_{n}(D))$ is discrete in $\C$.
This implies that $(\id_{\C^{n-1}}\times\exp)^{-1}(D)$ is a tame discrete set in $\C^{n}$ (cf. \cite[Theorem 3.9]{Rosay1988}).
Therefore $(\C^{n-1}\times\C^{*})\setminus\varphi(S)$ is a Zariski open Oka neighborhood of $\varphi(p)$ because its universal covering $\C^{n}\setminus(\id_{\C^{n-1}}\times\exp)^{-1}(D)$ is Oka (cf. \cite[Proposition 5.6.3]{Forstneric2017}).
It follows that the preimage $\varphi^{-1}(\C^{n-1}\times\C^{*})\setminus S\subset\C^{n}\setminus S$ is a Zariski open Oka neighborhood of $p$.
\end{proof}

\begin{proof}[Proof of Theorem \ref{theorem:blowup}]
Let $\pi:\Bl_{S\setminus S'}(\C^{n}\setminus S')\to\C^{n}\setminus S'$ denote the blowup of $\C^{n}\setminus S'$ along $S\setminus S'$.
As before, it suffices to prove that for any fixed point $p\in\Bl_{S\setminus S'}(\C^{n}\setminus S')$ there exists a Zariski open Oka neighborhood of $p$.
The argument in the proof of Theorem \ref{theorem:main} gives a holomorphic automorphism $\varphi\in\Aut\C^{n}$ such that $\varphi(\pi(p))\in\C^{n-1}\times\C^{*}$, $\varphi(S')\subset\C^{n-1}\times\{0\}$ and the discrete set $(\id_{\C^{n-1}}\times\exp)^{-1}(\varphi(S))\subset\C^{n}$ is tame.
Since the blowup of $\C^{n}$ along a tame discrete set is Oka (cf. \cite[Proposition 6.4.12]{Forstneric2017}), the blowup $\Bl_{(\id_{\C^{n-1}}\times\exp)^{-1}(\varphi(S))}\C^{n}$ is Oka.
Therefore the blowup $\Bl_{\varphi(S)\cap(\C^{n-1}\times\C^{*})}(\C^{n-1}\times\C^{*})$ covered by $\Bl_{(\id_{\C^{n-1}}\times\exp)^{-1}(\varphi(S))}\C^{n}$ is also Oka.
It follows that
\begin{align*}
\pi^{-1}(\varphi^{-1}(\C^{n-1}\times\C^{*}))=\Bl_{S\cap\varphi^{-1}(\C^{n-1}\times\C^{*})}\varphi^{-1}(\C^{n-1}\times\C^{*})\subset\Bl_{S\setminus S'}(\C^{n}\setminus S')
\end{align*}
is a Zariski open Oka neighborhood of $p$.
\end{proof}

%
%

\section{Proofs of the Corollaries}
\label{section:proof_corollary}

First, let us recall the definitions of ellipticity and its variants.

\begin{definition}[{cf. \cite[Definition 5.6.13]{Forstneric2017}}]
\label{definition:elliptic}
Let $Y$ be a complex manifold.
\begin{enumerate}[leftmargin=*]
	\item A \emph{spray} $(E,\pi,s)$ on $Y$ is a triple $(E,\pi,s)$ consisting of a holomorphic vector bundle $\pi:E\to Y$ and a holomorphic map $s:E\to Y$ such that $s(0_y)=y$ for each $y\in Y$.
	\item $Y$ is \emph{elliptic} (resp. \emph{subelliptic}) if there exists a spray $(E,\pi,s)$ (resp. a family $(E_j,\pi_j,s_j)\ (j=1,\ldots,k)$ of sprays) on $Y$ such that $\der s_{0_y}(E_y)=\T_y Y$ (resp. $\sum_{j=1}^k(\der s_{j})_{0_y}(E_{j,y})=\T_y Y$) for each $y\in Y$.
	\item $Y$ is \emph{weakly subelliptic} if for any compact set $K\subset Y$ there exists a family $(E_j,\pi_j,s_j)\ (j=1,\ldots,k)$ of sprays on $Y$ such that $\sum_{j=1}^k(\der s_{j})_{0_y}(E_{j,y})=\T_y Y$ for each $y\in K$.
\end{enumerate}
\end{definition}

As we mentioned, we prove the following stronger result than Corollary \ref{corollary:non-elliptic}.

\begin{corollary}
\label{corollary:non-weakly_subelliptic}
For any $n\geq 3$, the complement $\C^{n}\setminus((\overline{\N^{-1}})^{2}\times\{0\}^{n-2})$ is Oka but not weakly subelliptic.
\end{corollary}

In order to prove Corollary \ref{corollary:non-weakly_subelliptic}, we need to improve the result of Andrist, Shcherbina and Wold \cite[Theorem 1.1]{Andrist2016} as follows.
The proof is based on their idea and Gromov's method of composed sprays (cf. \cite[\S 6.3]{Forstneric2017}).

\begin{lemma}
\label{lemma:non-elliptic}
Let $Y$ be a Stein manifold of dimension at least three and $K\subset Y$ be a compact $\cO(Y)$-convex set with an infinite derived set.
Then $Y\setminus K$ is not weakly subelliptic.
\end{lemma}

\begin{proof}
To reach a contradiction, we assume that $Y\setminus K$ is weakly subelliptic.
Take a relatively compact open neighborhood $U\subset Y$ of $K$.
By assumption, there exists a family $(E_j,\pi_j,s_j)\ (j=1,\ldots,k)$ of sprays on $Y\setminus K$ such that $\sum_{j=1}^k(\der s_{j})_{0_y}(E_{j,y})=\T_y Y$ for each $y\in\partial U$.
By the Hartogs extension theorem for holomorphic vector bundles and sprays \cite[Theorem 1.2 and Theorem 4.1]{Andrist2016}, we can extend $E_j\ (j=1,\ldots,k)$ to holomorphic vector bundles $\widetilde\pi_j:\widetilde E_j\to Y\setminus A\ (j=1,\ldots,k)$, where $A\subset K$ is a finite set, and $s_j\ (j=1,\ldots,k)$ to holomorphic maps $\tilde s_j:\widetilde E_j\to Y\ (j=1,\ldots,k)$.
Note that $\tilde s_j(0_y)=y$ for each $y\in Y\setminus A$ by the identity theorem.

Let $B\subset Y\setminus A$ denote the closed complex subvariety of points $y\in Y\setminus A$ such that $\sum_{j=1}^k(\der\tilde s_{j})_{0_y}(\widetilde E_{j,y})\neq\T_y Y$ and set $S=A\cup B$.
Since $U$ is relatively compact and $B\cap\partial U=\emptyset$, the intersection $B\cap U\subset U\setminus A$ must be a discrete set.
Hence $K$ is not contained in $S$ and thus we may take a point $y_{0}\in\partial K\setminus S$.
Let us consider the fiber product and the associated maps
\begin{gather*}
E=\left\{(e_1,\ldots,e_k)\in\prod_{j=1}^k\left(\widetilde E_j\setminus\left(\widetilde\pi_j^{-1}(S)\cup\tilde s_j^{-1}(S)\right)\right):
\begin{array}{c}
	\tilde s_j(e_j)=\widetilde\pi_{j+1}(e_{j+1}), \\
	j=1,\ldots,k-1
\end{array}\right\}, \\
\pi:E\to Y\setminus S,\ (e_1,\ldots,e_k)\mapsto\widetilde\pi_{1}(e_1),\quad s:E\to Y\setminus S,\ (e_1,\ldots,e_k)\mapsto\tilde s_{k}(e_k).
\end{gather*}
Note that $\der s_{(0_{y},\ldots,0_{y})}(\T_{(0_{y},\ldots,0_{y})}\pi^{-1}(y))=\sum_{j=1}^k(\der\tilde s_{j})_{0_y}(\widetilde E_{j,y})=\T_y Y$ for each $y\in Y\setminus S$ (cf. \cite[Lemma 6.3.6]{Forstneric2017}).
Therefore the fiber preserving map $E\to (Y\setminus S)^2,\ e\mapsto(\pi(e),s(e))$ restricts to a holomorphic submersion from a neighborhood of the zero section $\{(0_y,\ldots,0_y):y\in Y\setminus S\}$ onto a neighborhood of the diagonal $\{(y,y):y\in Y\setminus S\}$.
Thus there exists an open neighborhood $V\subset Y\setminus S$ of $y_{0}$ such that for each $y\in V$ there exists $e\in E$ such that $\pi(e)=y$ and $s(e)=y_{0}$.
This contradicts to $y_{0}\in\partial K$ and $s(\pi^{-1}(Y\setminus K))\subset Y\setminus K$.
\end{proof}

\begin{proof}[Proof of Corollary \ref{corollary:non-weakly_subelliptic}]
Since $\overline{\N^{-1}}\subset\C$ is polynomially convex, $(\overline{\N^{-1}})^2\times\{0\}^{n-2}\subset\C^n$ is also polynomially convex.
Clearly the derived set of $(\overline{\N^{-1}})^{2}\times\{0\}^{n-2}\subset\C^{n}$ is infinite.
Thus Lemma \ref{lemma:non-elliptic} implies that its complement is not weakly subelliptic.

In order to prove that the complement is Oka, we use the localization principle for Oka manifolds (Theorem \ref{theorem:localization}).
Set $U_{j}=\C^{j-1}\times\C^{*}\times\C^{n-j}\ (j=1,\ldots,n)$ and $S=(\overline{\N^{-1}})^{2}\times\{0\}^{n-2}$.
Note that $\C^{n}\setminus S=\bigcup_{j=1}^{n}(U_{j}\setminus S)$ and each $U_{j}\setminus S$ is Zariski open in $\C^{n}\setminus S$.
By the localization principle, it suffices to show that each $U_{j}\setminus S$ is Oka.
For $j\geq 3$, $U_{j}\cap S=\emptyset$ and thus $U_{j}\setminus S=U_{j}$ is Oka.
Consider the exponential map $\pi=\exp\times\id_{\C^{n-1}}:\C^{n}\to U_{1}$.
Since $\pi$ is a covering map, $U_{1}\setminus S$ is Oka if and only if $\C^{n}\setminus\pi^{-1}(S)$ is Oka (cf. \cite[Proposition 5.6.3]{Forstneric2017}).
By definition, $\pi^{-1}(S)=\exp^{-1}(\N^{-1})\times(\overline{\N^{-1}})\times\{0\}^{n-2}$.
Note that $\pi^{-1}(S)$ is tame and its derived set $\pi^{-1}(S)'=\exp^{-1}(\N^{-1})\times\{0\}^{n-1}$ is discrete.
Therefore Theorem \ref{theorem:main} implies that $\C^{n}\setminus\pi^{-1}(S)$ is Oka and  thus $U_{1}\setminus S\cong U_{2}\setminus S$ are Oka.
\end{proof}

\begin{proof}[Proof of Corollary \ref{corollary:Hopf}]
Set $n=\dim Y>1$.
By Kodaira's argument \cite[Theorem 30]{Kodaira1966} (see also \cite{Hasegawa1993}), there exists a finite (unramified) covering map $(\C^{n}\setminus\{0\})/\langle\varphi\rangle\to\ Y$ where $\varphi\in\Aut\C^{n}$ is a holomorphic contraction, i.e. $\varphi(0)=0$ and $\varphi^{j}\to 0$ uniformly on compacts as $j\to\infty$.
Thus we may assume that $Y=(\C^{n}\setminus\{0\})/\langle\varphi\rangle$ from the beginning.
Let $\pi:\C^{n}\setminus\{0\}\to Y$ denote the quotient map.
Since $\varphi$ is a holomorphic contraction, there exists a holomorphic coordinate system on $\C^{n}$ in which $\varphi$ is lower triangular (cf. \cite[p.\,117]{Forstneric2017}).
In this coordinate system, the discrete set $\pi^{-1}(S)\setminus\overline{\B^n(0,1)}\subset\C^n$ projects to a discrete set in the first coordinate, and hence it is tame (cf. \cite[Theorem 3.9]{Rosay1988}).
Note that $(\pi^{-1}(S)\cap\overline{\B^n(0,1)})\cup\{0\}$ is compact.
Thus $\pi^{-1}(S)\cup\{0\}\subset\C^{n}$ is tame and $(\pi^{-1}(S)\cup\{0\})'=\{0\}$.
Therefore Theorem \ref{theorem:main} (resp. Theorem \ref{theorem:blowup}) implies the Oka property of the complement $(\C^{n}\setminus\{0\})\setminus \pi^{-1}(S)$ (resp. the blowup $\Bl_{\pi^{-1}(S)}(\C^{n}\setminus\{0\})$) and hence the complement $Y\setminus S$ (resp. the blowup $\Bl_{S}Y$) is Oka.
\end{proof}

%
%

\section*{Acknowledgement}

I wish to thank my supervisor Katsutoshi Yamanoi for many constructive comments, and Franc Forstneri\v{c} for helpful remarks.
This work was supported by JSPS KAKENHI Grant Number JP18J20418.

%
%


\begin{thebibliography}{10}


\bibitem{Andrist2016}
R.~B. Andrist, N.~Shcherbina, and E.~F. Wold.
\newblock The {H}artogs extension theorem for holomorphic vector bundles and
  sprays.
\newblock {\em Ark. Mat.}, 54(2):299--319, 2016.


\bibitem{Andrist2019}
R.~B. Andrist and R.~Ugolini.
\newblock A new notion of tameness.
\newblock {\em J. Math. Anal. Appl.}, 472(1):196--215, 2019.


\bibitem{Forstneric2013}
F.~Forstneri\v{c}.
\newblock Oka manifolds: from {O}ka to {S}tein and back.
\newblock {\em Ann. Fac. Sci. Toulouse Math. (6)}, 22(4):747--809, 2013.
\newblock With an appendix by Finnur L\'{a}russon.


\bibitem{Forstneric2017}
F.~Forstneri\v{c}.
\newblock {\em Stein manifolds and holomorphic mappings}, volume~56 of {\em
  Ergebnisse der Mathematik und ihrer Grenzgebiete. 3. Folge. A Series of
  Modern Surveys in Mathematics [Results in Mathematics and Related Areas. 3rd
  Series. A Series of Modern Surveys in Mathematics]}.
\newblock Springer, Cham, second edition, 2017.
\newblock The homotopy principle in complex analysis.


\bibitem{Forstneric2002}
F.~Forstneri\v{c} and J.~Prezelj.
\newblock Oka's principle for holomorphic submersions with sprays.
\newblock {\em Math. Ann.}, 322(4):633--666, 2002.


\bibitem{Gromov1989}
M.~Gromov.
\newblock Oka's principle for holomorphic sections of elliptic bundles.
\newblock {\em J. Amer. Math. Soc.}, 2(4):851--897, 1989.


\bibitem{Hasegawa1993}
K.~Hasegawa.
\newblock Deformations and diffeomorphism types of {H}opf manifolds.
\newblock {\em Illinois J. Math.}, 37(4):643--651, 1993.


\bibitem{Kodaira1966}
K.~Kodaira.
\newblock On the structure of compact complex analytic surfaces. {II}.
\newblock {\em Amer. J. Math.}, 88:682--721, 1966.


\bibitem{Kusakabe2019}
Y.~Kusakabe.
\newblock Elliptic characterization and localization of {O}ka manifolds.
\newblock {\em Indiana Univ. Math. J.}, to appear.


\bibitem{Rosay1988}
J.-P. Rosay and W.~Rudin.
\newblock Holomorphic maps from {${\bf C}^n$} to {${\bf C}^n$}.
\newblock {\em Trans. Amer. Math. Soc.}, 310(1):47--86, 1988.


\bibitem{Winkelmanna}
J.~Winkelmann.
\newblock Tame discrete sets in algebraic groups.
\newblock arXiv:1901.08952.


\bibitem{Winkelmann2019}
J.~Winkelmann.
\newblock Tame discrete subsets in {S}tein manifolds.
\newblock {\em J. Aust. Math. Soc.}, 107(1):110--132, 2019.


\end{thebibliography}

\end{document}